\documentclass{amsart}
\usepackage[english]{babel}
\usepackage[utf8]{inputenc}

\usepackage{amsmath,amssymb,amsthm}
\usepackage{enumitem}

\usepackage{fullpage}
\usepackage{enumitem}
\usepackage[all]{xy}
\usepackage{hyperref}

\usepackage{tikz}

\usepackage{xcolor}

\usepackage[all]{xy}

\newtheorem{thm}{Theorem}[section]    % Standard theorem environment
\newtheorem{lem}[thm]{Lemma}          % Lemma environment with numbering
\newtheorem{prop}[thm]{Proposition}
\newtheorem{cor}[thm]{Corollary}
\newtheorem*{open}{Question}

\theoremstyle{definition}
\newtheorem{defi}[thm]{Definition}

\newtheorem{rem}[thm]{Remark}             % Unnumbered environment for remarks.
\newtheorem*{ex}{Example}
\newcommand{\Db}{\mathbb{D}}

\newcommand{\Qb}{\mathbb{Q}}
\newcommand{\Wb}{\mathbb{W}}

\newcommand{\Cc}{\mathcal{C}}
\newcommand{\Ec}{\mathcal{E}}

\renewcommand{\Mc}{\mathcal{M}}
\newcommand{\Pc}{\mathcal{P}}

\newcommand{\Aut}{\text{Aut}}

\newcommand{\Id}{\text{Id}}

\newcommand{\Graphs}{{{\mathcal G}raphs}}
\newcommand{\InjGraphs}{{\mathcal{I}nj{\mathcal G}raphs}}
\newcommand{\Groups}{{{\mathcal G}roups}}
\newcommand{\Monoids}{{{\mathcal M}onoids}}
\newcommand{\Part}{{{\mathcal P}art}}
\newcommand{\Posets}{{{\mathcal P}osets}}
\newcommand{\kAlg}{{{\mathcal A}lg_k}}
\newcommand{\kEvAlg}{{{\mathcal{E}v{\mathcal A}lg_k}}}
\newcommand{\kRegEvAlg}{{{\mathcal{R}eg\mathcal{E}v{\mathcal A}lg_k}}}

\begin{document}

\title{A remark on the number of automorphisms of some algebraic structures}

%    Information for first author
\author{R\'emi Molinier}
\address{Univ. Grenoble Alpes, CNRS, IF, 38000 Grenoble, France}
\email{remi.molinier@univ-grenoble-alpes.fr\\
ORCID: \href{https://orcid.org/0000-0002-3742-5307}{0000-0002-3742-5307}}
%    \thanks will become a 1st page footnote.
%\thanks{The second author was partially supported by ???}


\begin{abstract}
In these notes we look at the following question: given a category $\mathcal C$ of algebraic structure (e.g. the category of groups, monoids, partial groups, ...)  and a rational $r\in \mathbb Q$, does there exists an element $x\in \mathcal C$ such that the size of its automorphism group $\text{Aut}_{\mathcal C} (x)$ divided by the size of $x$ (whatever that would means) is equal to $r$ ?
To our knowledge, this question was introduced by Tărnăuceanu in the category of groups. Here, we answer positively to this question in the categories of evolution algebras, graphs, monoids, partial groups and posets. 
\end{abstract}

\maketitle

\section{Introduction}

In a recent paper \cite{Ta}, Tărnăuceanu proved the following nice density result. 
\begin{thm}[{\cite{Ta}, Main Theorem}]
	The set 
\[\left\{\frac{|\Aut_{\Groups}(G)|}{|G|}\;\middle|\; G \text{ a finite group}\right\}\]
is dense in $\Qb\cap]0,+\infty[$. 
\end{thm} 

His result is actually a bit stronger since he proved that the subset considering only finite abelian groups is dense. 
From this Theorem, he addresses the following question. 

\begin{open}
Does there exists, for every $r\in \Qb\cap ]0;+\infty[$, a group $G$ such that $\dfrac{|\Aut_{\Groups}(G)|}{|G|}=r$ ?
\end{open}

We are not trying to answer this question here but we look at the question in two more general contexts: the category $\Monoids$ of monoids and the category $\Part$ of partial groups. Theses two algebraic structures give a generalisation of the structure of group, one removing the axiom about the existence of inverses and the other by relaxing the domain of the product (i.e. they can be tuples of elements for which the product is not define). In these two categories, it is much more easier to find an object with a given number of automorphisms. More precisely, if $\Cc$ is the category $\Monoids$ or $\Part$,  for every group $G$, one can find infinitely many objects of $\Cc$, with $G$ as automorphism group. This realisability result for partial group is due to the author in collaboration with Diaz and Viruel \cite{DMV}. The corresponding one on monoids is not new and can be deduced from the well known construction highlighted in section \ref{sec:graph to mon} bellow, and Theorems from Frucht \cite{Fr} and Sabidussi \cite{Sa} on the realisability for graphs which state that every group is isomorphic to the automorphism group of a well chosen graph. This is actually the same tactic that is use in \cite{DMV}: we first prove that for every graph $\Gamma$ we can find a partial group $\Mc$ with the same automorphism group as $\Gamma$ and we conclude using the results of Frucht and Sabidussi.

Here, after giving functors from the category of graphs to $\Monoids$ and $Part$  which induce isomorphisms on automorphism groups, we consider well chosen graphs to obtain the followings.

\begin{thm}[cf Theorem \ref{thm:main monoids}]\label{thm:intro mon}
Let $r\in\Qb\cap]0,+\infty[$. 

There exists a monoid $M$ such that $r=\dfrac{|\Aut_\Monoids(M)|}{|M|}$.
Moreover, the monoid $M$ can be chosen to be commutative. 
\end{thm}  

\begin{thm}[cf Theorem \ref{thm:main part}]\label{thm:intro part}
Let $r\in\Qb\cap]0,+\infty[$. 

There exists a partial group $\Mc$ such that $r=\dfrac{|\Aut_\Part(\Mc)|}{|\Mc|}$.
\end{thm}

Note that the construction used here for partial groups is different from the one used in \cite{DMV}. Indeed, in \cite{DMV}, the partial groups considered are infinite even when the graph is finite (see for instance \cite[Example 6.4]{DMV}). Hence the construction given in Section \ref{subsec:graph to part} leads to another proof of the universality of the category of partial groups.

We also use the same tactic on posets to get the same kind of result in that context.

\begin{thm}[cf Theorem \ref{thm:main posets}]\label{thm:intro poset}
Let $r\in\Qb\cap]0,+\infty[$. 

There exists a poset $P$ such that $r=\dfrac{|\Aut_\Posets(P)|}{|P|}$.
\end{thm}

As mention above, the main idea is to consider functors that preserve automorphism groups from the categories of these algebraic structures to the category of graphs. Graphs are easy to play with and one can easily modify a graph without changing its automorphism group. By considering well chosen graphs, we get similar results for graphs where the "size" of a graph is either the number of vertices or the number of vertices plus the number of edges. The previous results are then consequences of the following.

\begin{thm}[cf Corollary \ref{cor:main graphs} and Corollary \ref{cor:main graphs2}]\label{thm:intro graph}
Let $r\in\Qb\cap]0,+\infty[$. 
\begin{enumerate}
\item There exists a graphs $\Gamma$ such that $r=\dfrac{|\Aut_\Graphs(\Gamma)|}{|V|}$ where $V$ is the set of vertices of $\Gamma$.
\item There exists a graphs $\Gamma$ such that $r=\dfrac{|\Aut_\Graphs(\Gamma)|}{|V|+|E]}$ where $V$ and $E$ are the sets of, respectively, vertices and edges of $\Gamma$.
\end{enumerate}
\end{thm}

Finally, as an other application of the machinery developed here (and a realisability result from \cite{CLTV}), we get a somewhat similar results for finite dimensional evolution algebras (i.e. vector spaces with a bilinear inner product) over any field $k$ where we understand "size" as dimension.

\begin{thm}[cf Theorem \ref{thm: main evoalg}]\label{thm:intro evoalg}
Let $k$ be a field, and $r\in\Qb\cup ]0,\infty[$. Then there exists a
regular evolution $k$-algebra $A$ such that $r=\frac{|\Aut_{\kAlg}(A)|}{\dim_k(A)}$.
\end{thm}

\subsection*{Organisation of the paper}
In Section \ref{sec:algebraic structure} we specified the categories of graphs and algebraic structures we are working with here. We are more expansive on partial groups as it is less common. In Section \ref{sec:algebraic structure to graph} we gives explicite functors that preserve automorphims groups from these categories of algebraic structures to the category of graphs. Section \ref{sec:graphs} gives the graphs that will be the key ingredients, describes their automorphisms and how we get Theorem \ref{thm:intro graph} above. Finally, theorems \ref{thm:intro mon}, \ref{thm:intro part}, \ref{thm:intro poset} and \ref{thm:intro evoalg}  are proven in Section \ref{sec:applications} combining the results from Sections \ref{sec:algebraic structure to graph} and \ref{sec:graphs}.

%Frutch \cite{Fr} proved in 1939 that any finite group is the group of automorphism a finite graph. The generalization to infinite group is due, independently to de Groot \cite{Gr} and Sabidussi \cite{Sa} in 1959 and 1960 respectively.
%
%\begin{thm}%[\bf Frucht's (for the finite case), de Groot-Sabidussi (for the inifinite case)]
%Let $G$ be a group. There exists a graph $\Gamma$ such that $G$ is isomorphic to the group of automorphisms of $\Gamma$.
%\end{thm}
%
%
%This was already answered positivilly in \cite{DMV}. 

\subsection*{Notations}
In these notes, if $\Cc$ is a category, for every object $x$ of $\Cc$, we will denote by $\Aut_\Cc(x)$ the group of automorphisms of $x$ in $\Cc$. Also if $X$ is a finite set $|X|$ will denote the size of $X$ and $\Wb(X)$ will denote the free monoid of the finite words on $X$, where the elements will be denoted by a list of elements of $X$,  and for two words $u,v\in\Wb(X)$, $u\circ v$ will denote the concatenation of $u$ and $v$. We often identify $X$ as the subset $\{(x)\mid x\in X\}$ of words of length 1 in $\Wb(X)$. Given two sets $X$ and $Y$ and a map $\varphi\colon X\to Y$ we will denote by $\Wb(\varphi)\colon \Wb(X)\to\Wb(Y)$ the induce map by $\varphi$ defined, for all $u=(x_1,x_2,\dots,x_n)\in\Wb(X)$, by $\Wb(\varphi)(u)=(\varphi(x_1),\varphi(x_2),\dots,\varphi(x_n))$.
For $k$ a field and $V$ a $k$-vector space of finite dimension, we will denote by $\dim_k(V)$ its dimension.
Finally, for $p\geq 1$, $\Sigma_p$ will denote the symmetric group on $\{1,2,\dots,p\}$ which is of order $p!$. 

\subsection*{Acknowledgement} The author is thankful to Antonio Viruel for pointing out a mistake in the first draft of this paper and suggesting the application to evolution algebras.

\section{recalls and notations on the algebraic structures}\label{sec:algebraic structure}
We recall here the algebraic structures and give the associated notations that we use in this paper.

\subsection{Evolution algebra}
Given a field $k$, a $k$-\emph{algebra} is a $k$-vector space $A$ together with a $k$-bilinear inner product $A\times A\to A$. Given two $k$-algebras $A$ and $B$, a \emph{algebra homomorphism} of $k$-algebras from $A$ to $B$ is a linear map from $A$ to $B$ which commutes with inner product. With this notion of morphism, the class of $k$-algebras defines a category $\kAlg$.

A $k$-algebra $A$ will be called an \emph{evolution algebra} if it exists a basis $B=\{b_i\mid i\in I\}$ such that $b_ib_j=0$ whenever $i\neq j$. Such a basis will be called a \emph{natural basis} of $A$. Finally, an evolution algebra will be called a \emph{regular evolution algebra} if moreover $A^2=A$. Evolution algebra are commutative but possibly not assoiative algebras that were introduced by Tian \cite{Ti} to study reproduction in non-Mendelian genetic. We will denote by $\kEvAlg$, respectively $\kRegEvAlg$, the full subcategories of evolution $k$-algebras and regular evolution $k$-algebras.

\subsection{Graphs}
We work here with the category $\Graphs$ of graphs which are undirected and simple (without loops or multiple edges). More precisely, a \emph{graph} is a couple $(V,E)$ with $V$ a set, called the set of \emph{vertices} and $E$ a set of subsets of $V$ of size 2, called the set of \emph{edges}. A vertex $x\in V$ will be called \emph{isolated} if there is no edge of $E$ containing $x$.
 A \emph{graph homomorphism} from a graph $\Gamma_1=(V_1,E_1)$ to a graph $\Gamma_2=(V_2,E_2)$ is then a map $f\colon V_1\to V_2$ such that, for any $e\in E_1$, $\overline{f}(e)\in E_2$ where $\overline{f}\colon \Pc(V_1)\to \Pc(V_2)$ is the mapping, induced by $f$, from $\Pc(V_1)$ the set of all subsets of $V$ to $\Pc(V_2)$ the set of all subsets of $V_2$. Finally, we will denote by $\InjGraphs$ the subcategory of $\Graphs$ with the same objects but with only injective graph homomorphisms. Notice that for any graph $\Gamma$, $\Aut_\InjGraphs (\Gamma)=\Aut_\Graphs (\Gamma)$ since every automorphism is injective.  

\subsection{Monoids}
Recall that a \emph{monoid} is a couple $(M,\ast)$ where $M$ is a set and $\ast\colon M\times M\to M$ is a binary operation which is associative and has a neutral element. A \emph{monoid homomorphism} $\varphi\colon (M_1,\ast_1)\to(M_2,\ast_2)$ is a mapping such that for all $m,n\in M_1$, $\varphi(m\ast_1 n)= \varphi(m)\ast_2\varphi(n)$ and such that $\varphi(e_1)=e_2$ where $e_1$ and $e_2$ are the neutral elements of $M_1$ and $M_2$ respectively. With this notion of morphism, the class of monoids defines a category $\Monoids$. The category $\Groups$ of groups is the full subcategory of $\Monoids$ consisting of the monoids $(M,\ast)$ such that every element of $M$ has an inverse.

\subsection{Partial groups}
The notion of partial group is less common and we give slightly more details. We refer the reader to \cite[Section 2]{Ch0} or \cite[Section 1]{Ch1} for more comprehensive introductions about this structure.
\medskip

\begin{defi}\label{def:PG}
Let $\Mc$ be a set and $\Db\subseteq \Wb(\Mc)$ be a subset such that,
\begin{enumerate}[label=(D\arabic*)]
\item\label{cond:D1} $\Mc\subseteq \Db$; and
\item\label{cond:D2} $u\circ v\in\Db\Rightarrow u,v\in\Db(\Mc)$ (in particular, $\emptyset\in\Db(\Mc)$).
\end{enumerate}
A mapping $\Pi:\Db\rightarrow\Mc$ is a \emph{product} on $\Mc$ if
\begin{enumerate}[label=(P\arabic*)]
\item\label{cond:P1} $\Pi$ restricts to the identity on $\Mc$; and
\item\label{cond:P2} if $u\circ v\circ w\in \Db(\Mc)$ then $u\circ \Pi(v)\circ w\in\Db(\Mc)$ and
\[\Pi(u\circ v\circ w)=\Pi\left(u\circ\Pi(v)\circ w\right).\]
\end{enumerate}
The \emph{unit} of $\Pi$ is then defined as $\Pi(\emptyset)$ that we will denote $1_\Mc$ or $1$ when there is no ambiguity.
A \emph{partial monoid} is a triple $(\Mc,\Db,\Pi)$ where $\Pi$ is a product defined on $\Db(\Mc)$.

An \emph{inversion} on $\Mc$ is an involutory bijection $x\mapsto x^{-1}$ on $\Mc$ together with the induced mapping  $u\mapsto u^{-1}$ on $\Wb(\Mc)$ defined by,
\[u=(x_1,x_2,\dots,x_n)\mapsto (x_n^{-1},x_{n-1}^{-1},\dots,x_1^{-1}).\]
A \emph{partial group} is then a tuple $\left(\Mc,\Db,\Pi,(-)^{-1}\right)$ where $\Pi$ is a product on $\Db(\Mc)$ and $(-)^{-1}$ is an inversion on $\Mc$ satisfying
\begin{enumerate}[resume,label=(P\arabic*)]
\item\label{cond:P3} if $u\in\Db$ then $(u^{-1},u)\in\Db$ and $\Pi(u^{-1}\circ u)=1$.
\end{enumerate}
\end{defi}

To simplify the notation, we will use $(\Mc,\Db)$ or even juste $\Mc$ to refer to a partial group $\left(\Mc,\Db,\Pi,(-)^{-1}\right)$ when the rest of the data is understood or not used. Also, a word $w\in\Db$ will be called \emph{non-degenerate} if there is no $1$ in $w$. Notice that, thanks to axiom \ref{cond:P2} where we can take $v$ to be empty, $\Db$ is totally determined by its non-degenerated word.

\begin{ex}
 Let $(\Mc,\Db)$ be a partial group.  if the domain $\Db=\Wb(\Mc)$ then $\Mc$ is a group via the binary operation $(x,y)\in\Mc^2\mapsto \Pi(x,y)\in\Mc$. 
\end{ex}

%\begin{ex}
%Let $\Fc(a)=\{1,a,a^{-1}\}$ and set the non-degenerated words of $\Db\left(\Fc(a)\right)$ to be all possible words alternating $a$'s and $a^{-1}$'s. In other words, the non-degenerated words of $\Db\left(\Fc(a)\right)$ are all the different finite sub-words of the infinite words $(a,a^{-1},a,a^{-1},a,a^{-1},\dots)$. The inversion is understood  and, for any word $u\in\Db\left(\Fc(a)\right)$,
%\[
%\Pi(u)=
%\begin{cases}
%1 &\text{if the number of $a$'s equal the number of $a^{-1}$'s},\\
%a &\text{if the number of $a$'s exceed the number of $a^{-1}$'s (necessarly by 1)},\\
%a^{-1} &\text{if the number of $a^{-1}$'s exceed the number of $a$'s (necessarly by 1)}.
%\end{cases}
%\]
%One can then check that $\left(\Fc(a),\Db(\Fc(a)),\Pi_a,(-)^{-1}\right)$ defines a partial group.
%\end{ex}
%
%This last example is actually the \emph{free partial group on the set $\{a\}$} as it is detailed in \cite[Lemma 1.12]{Ch1}.  Till the end of the paper, We will keep the notation $\left(\Fc(a),\Db(\Fc(a)),\Pi_a,(-)^{-1}\right)$ for the free partial group on $\{a\}$.

\begin{defi}
Let $\left(\Mc_1,\Db(\Mc_1),\Pi_1,(-)^{-1}\right)$ and $\left(\Mc_2,\Db(\Mc_2),\Pi_2,(-)^{-1}\right)$ be two partial group.
A \emph{partial groups homomorphism} is a map $\varphi\colon \Mc_1\to \Mc_2$ such that 
\begin{enumerate}[label=(H\arabic*)]
\item $\Wb(\varphi)\left(\Db(\Mc_1)\right)\subseteq \Db(\Mc_2)$;
\item for any $u\in\Db(\Mc_1)$, $\Pi_2\left(\Wb(\varphi)(u)\right)=\varphi\left(\Pi_1(u)\right)$.
\end{enumerate}
\end{defi}

With this notion of homomorphisms and the usual composition on maps, we have a category $\Part$ which contains $\Groups$ as a full subcategory.

\subsection{Posets}
We will also work here with posets. Recall that a \emph{poset} is a couple $(X,\leq)$ where $X$ is a set and $\leq$ a partial order, i.e. a homogeneous relation which is reflexive, antisymmetric, and transitive. A \emph{poset homomorphism} $\varphi\colon (X_1,\leq_1)\to (X_2\leq_2)$ is then a mapping $\varphi\colon X_1\to X_2$ which is order-preserving, i.e. for all $x,y\in X_1$, if $x\leq_1 y$, then $\varphi(x)\leq_2\varphi(y)$. Endowed with this notion of morphism, the class of posets defines a category that we will denote by $\Posets$. 
\section{From graphs to algebra, monoids, partial groups and posets}\label{sec:algebraic structure to graph}

We give here the functors that we will use to prove the main results.

\subsection{From graphs to evolution algebras}

The following construction is due to Costaya, Ligouras, Tocino and Viruel \cite{CLTV}.

Let $\Gamma=(V,E)$ be a graph. We define $A(\Gamma)$ as the evolution algebra with natural basis $B=E\cup V$ and where the inner product is given on the natural basis by, for all $v\in V$ and 
\[ v^2=v \qquad \text{and} \qquad e^2=e+\sum_{u\in e}u.\]

This gives a regular evolution algebra as shown in \cite[Lemma 3.3]{CLTV} and it defines a faithul functor
\[A\colon \Graphs \to \kRegEvAlg.\]

\begin{thm}[{cf. \cite[Theorem 3.1]{CLTV}}]\label{thm:graph to evoalg}
et $\Gamma=(V,E)$ be a graph. 
Then $\dim_k(A(\Gamma))=|E|+|V|$ and  $\Aut_{Graphs}(\Gamma)\cong\Aut_{\kAlg}(A(\Gamma))$. 
\end{thm}

\begin{proof}
By construction, $A(\Gamma)$ has a basis of size $|E|+|V|$ and so $\dim_k(A(\Gamma))=|E|+|V|$ and the isomrophism on automorphism groups is a consequences of \cite[Lemma 3.7]{CLTV}.
\end{proof}

\begin{rem}\label{rmk:realisability evoalg}
In \cite{SZ}, Sriwongsa and Zou give another construction of a regular evolution $k$-algebra $E$ from a graph $\Gamma=(V,E)$ with automorphism group isomorphic to the one of $\Gamma$ but this time the dimension of $E$ is $|V|$ instead of $|V|+|E|$. 
\end{rem}

\subsection{From graphs to monoids}\label{sec:graph to mon}

The following construction is well known but the author couldn't find a reference in the literature. However there is a discussion on math stack exchange \cite{monoid} highlighting it.

Let $\Gamma=(V,E)$ be a graph. Set $M(\Gamma)$ be the union of $V$, $E$ and two arbitrary elements denoted $0$ and $1$ 
\[M(\Gamma)=V\sqcup E\sqcup \{0,1\}.\]
We then define a product $\ast\colon M(\Gamma)\times M(\Gamma)\to M(\Gamma)$ setting, $0$ as a absorbing element on $M(\Gamma)\setminus\{1\}$ (i.e. for all $\alpha\in M(\Gamma)\setminus \{1\}$, $0\ast \alpha=\alpha\ast 0=0$), $1$ as a neutral element (i.e. for all $\alpha\in M(\Gamma)$, $1\ast\alpha=\alpha\ast 1=\alpha$) and for all $x,x_1,x_2\in V$, $e,e_1,e_2\in E$
\[x_1\ast x_2=\begin{cases} x_1 &\text{if $x_1=x_2$}\\ 0&\text{else}\end{cases}\qquad\qquad x\ast e=e\ast x=\begin{cases} e &\text{if $x\in e$}\\ 0&\text{else,}\end{cases}\qquad\qquad  e_1\ast e_2=0.\]

\begin{prop}
Let $\Gamma=(V,E)$ be a graph. The couple $(M(\Gamma),\ast)$ is a monoid and it yields a functor 
\[M\colon \InjGraphs \to \Monoids.\]
\end{prop}

\begin{proof}
The fact that $(M(\Gamma),\ast)$ is a monoid is a direct consequence of the definition and left to the reader. If $f\colon \Gamma_1=(V_1,E_1)\to \Gamma_2=(V_2,E_2)$ is a, injective graph homomorphism, it sends a vertex to a vertex, an edge to an edge and, for all $x\in V_1$ and $e\in E_1$, $x\in e_1$ if and only if $f(x)\in \overline{f}(e_1)$. Thus it induces a well define monoid homomorphism $\overline{f}\colon (M(\Gamma_1)\to(M(\Gamma_1)$ (which must sends $0$ to $0$ and $1$ to $1$). This construction preserves composition and thus defines a functor $M\colon \Graphs \to \Monoids$.
\end{proof}	

Notice that we have to work with injective graph homomorphism. Indeed, if one consider $\Gamma_1$ to be a graph with two isolated vertices $x_1$ and $x_2$ and no edges, $Gamma_2$ to be the graph with only one vertex $x_0$ and $f\colon \Gamma_1\to \Gamma_2$ to be the unique map from $\Gamma_1$ to $\Gamma_2$ then we should have $1x_0=x_0\ast x_0=f(x_1)\ast f(x_2)=f(x_1\ast x_2)=f(x_1)\ast f(x_2)=f(0)=0$ which give then a contradiction.

The following Lemma is also a direct consequence of the definition of $M(\Gamma)$.
\begin{lem}\label{lem:properties of M(Gamma)}
Let $\Gamma=(V,E)$ be a graph.
\begin{enumerate}[label=(\roman*)]
\item\label{item:properties of M(Gamma)(i)} $M(\Gamma)$ is commutative.
\item\label{item:properties of M(Gamma)(ii)} For all $\alpha\in M(\Gamma)\setminus\{0,1\}$, $\alpha\ast \alpha=\alpha$ if and only if $\alpha\in V$.
\item\label{item:properties of M(Gamma)(iii)} For all $e\in M(\Gamma)\setminus\{0\}$, $\alpha\ast \alpha=0$ if and only if $\alpha\in E$. 
\item\label{item:properties of M(Gamma)(iv)} For all $x\in V$ and $e\in E$, $x\ast e\neq 0$ if and only if $x\in e$.
\item\label{item:properties of M(Gamma)(v)} If $\Gamma$ is finite then so is $M(\Gamma)$ and $|M(\Gamma)|=|V|+|E|+2$.
\end{enumerate}
\end{lem}

We now look at isomorphisms. 
\begin{lem}\label{lem:mon to graph}
Let $\Gamma_1=(V_1,E_1)$ and $\Gamma_2=(V_2,E_2)$ be graphs and $\varphi\colon M(\Gamma_1)\to M(\Gamma_2)$ be an monoid isomorphism. 
\begin{enumerate}[label=(\alph*)]
\item\label{item:mon to graph (a)} $\varphi(0)=0$, $\varphi(V_1)=V_2$ and $\varphi(E_1)=E_2$.
\item\label{item:mon to graph (b)} Let $x,y\in V_1\subseteq M(\Gamma_1)$ with $x\neq y$. If $\{x,y\}\in E_1$ then $\{\varphi(x),\varphi(y)\}\in E_2$ and $\varphi(\{x,y\})=\{\varphi(x),\varphi(y)\}$. 
\end{enumerate}
\end{lem}

\begin{proof}
Notice that we have $\varphi(1)=1$ (as $\varphi$ is a monoid homomorphism) and $\varphi(0)=0$ since $0$ is the only absorbing element on $M(\Gamma_i)\setminus \{1\}$ for $i\in\{1,2\}$.
Moreover, a monoid homomorphism sends an idempotent to an idempotent, thus by Lemma \ref{lem:properties of M(Gamma)}\ref{item:properties of M(Gamma)(ii)} we have $\varphi(V_1)=V_2$ (considering $\varphi$ and its inverse). finally, as $\varphi$ is bijective, \[\varphi(E_1)=\varphi\left(M(\Gamma_1)\setminus \left(V_1\cup \{0,1\}\right)\right)=M(\Gamma_2)\setminus \varphi\left(V_1\cup \{0,1\}\right)=  M(\Gamma_2)\setminus \left(V_1\cup \{0,1\}\right)=E_2\]
which conclude the proof of item \ref{item:mon to graph (a)}.

Now let  $x,y\in V_1\subseteq M(\Gamma_1)$ with $x\neq y$ such that $e=(x,y)\in E_1$. by \ref{item:mon to graph (a)}, $\varphi(e)\in E_2$. Moreover, since  $x\ast e=e$ and $y\ast e=e$, we have that $\varphi(x)\ast \varphi(e)=\varphi(e)$ and $\varphi(y)\ast \varphi(e)=\varphi(e)$, and by Lemma \ref{lem:properties of M(Gamma)}\ref{item:properties of M(Gamma)(iv)}, $\{\varphi(x),\varphi(y)\}\subseteq \varphi(e)$. Finally, since $\varphi$ is injective and $x\neq y$, we have $\varphi(x)\neq\varphi(y)$ and thus $\varphi(e)=\{\varphi(x),\varphi(y)\}$.
\end{proof}

\begin{thm}\label{thm:graph to mon}
Let $\Gamma=(V,E)$ be a graph. Then $\Aut_\Graphs(\Gamma)\cong\Aut_\Monoids(M(\Gamma))$.
\end{thm}

\begin{proof}
Since $M\colon \InjGraphs \to \Monoids$ is a functor, $M$ induces a group homomorphism $\Phi\colon \Aut_\Graphs(\Gamma)=\Aut_\InjGraphs(\Gamma)\to\Aut_\Monoids(M(\Gamma))$. Moreover, notice that $f=M(f)|_V$ and thus, if $M(f)$ is the identity, so is $f$. Hence $\Phi$ is injective. 
Finally, by Lemma \ref{lem:mon to graph}, if $\varphi\in \Aut_\Monoids(M(\Gamma))$, then $\varphi|_V$ defines a graph homomorphism from $\Gamma$ to itself which is an automorphism of $\Gamma$ and we have $\Phi(\varphi|_V)=\varphi$. Hence $\Phi$ is an isomorphism.  
\end{proof}

\subsection{From graphs to partial groups}\label{subsec:graph to part}

Let $\Gamma=(V,E)$ be a graph. Set $\Ec(\Gamma)$ as the subset of $\Wb(V)$ consisting in the empty word together with the words of length 1 and the words of length 2 corresponding to edges in $E$ :
\[\Ec(\Gamma)=\{\emptyset\}\cup \{(x)\mid x\in V\}\cup \{(x,y)\mid \{x,y\}\in E\}\subset \Wb(V).\]
Here we decided to not identify $V$ with the words of length one in $\Wb(V)$ to help the reader when considering words as we will play with elements of $\Wb(V)$ and elements of $\Wb(\Ec(\Gamma))\subseteq \Wb(\Wb(V))$. 

The domain $\Db(\Gamma)$ will be define as the set of words $u=(u_1,u_2,\dots,u_r)\in \Wb(\Ec(\Gamma))$ so that the concatenation $u_1\circ u_2,\circ\cdots\circ u_r\in \Wb(V)$ is a word of the form $(x,x,\dots,x)$ for $x\in V$ (in particular, for all $i\in\{1,2\dots,r\}$, $u_i=(x)$ or $u_i=\emptyset$) or a finite subword of an infinite word of the form 
\[(\underbrace{x,x,\dots,x}_{2k_1\text{ terms}},\underbrace{y,y,\dots,y}_{2r_1\text{ terms}},\underbrace{x,x,\dots,x}_{2k_2\text{ terms}},\underbrace{y,y,\dots,y}_{2r_2\text{ terms}},\dots)\]
with $\{x,y\}\in E$ and for all $j$, $k_j,r_j\geq 1$. For example, if $\{x,y\}\in E$, $(x,x,x,y,y,y,y,y,x,x,y,y,y,y,y,y,x)\in\Db(\Gamma)$.
%\begin{enumerate}[label=(\Roman*)]
%\item $u\in\Wb(x_v)\smallsetminus\{\emptyset\}$ for $v\in V$;
%\item\label{case:II} $u$ is a finite sub-words of the infinite words $(a_v^w,a_w^v,a_v^w,a_w^v,a_v^w,a_w^v,\dots)$ for $\{v,w\}\in E$;
%\item\label{case:III} there exists $e=(v,w)\in E$ such that $u$ is a finite sub-word of the infinite word $(x_{v},x_{w},x_{w},x_{v},x_{v},x_{w},\dots)$;
%\item\label{case:IV} there exists $e=(v,w)\in E$ such that $u$ is obtained from a finite sub-word of $(x_{v},x_{w},x_{w},x_{v},x_{v},x_{w},\dots)$ by replacing some sub-words $(x_{v},x_{w})$ by the word $(a_v^w)$ and some sub-words $(x_{w},x_{v})$ by $(a_w^v)$. 
%\end{enumerate}

The product $\Pi\colon\Db(\Gamma)\to \Ec(\Gamma)$ is then given, for $u=(u_1,u_2,\dots,u_r)\in \Db(\Gamma)$, by $\Pi(u)$ as the word obtained from the concatenation $u_1\circ u_2,\circ\cdots\circ u_r$ by removing subwords of the form $(x,x)$ for $x\in V$ until no such reduction is possible (one can check that the resulting word is independent of the reduction process). In particular, since by definition of $\Db(\Gamma)$ $u$ involves at most two vertices, and two vertices of the same edge when it involves exactly two, we have that $\Pi(u)\in\Ec(\Gamma)$. Notice also that the empty word plays here the role of the unit. %and we will denote it by 1 from now on.

%\begin{enumerate}[label=(\roman*)]
%\item If $u=(u_1,u_2,\dots,u_r)$ so that the concatenation $u_1\circ u_2,\circ\cdots\circ u_r=(\underbrace{v,v,\dots,v}_{\text{$r$ times}})$ then 
%\[\Pi(u)=
%\begin{cases}
%\emptyset&\text{if $r$ is even, and}\\
%v&\text{if $r$ i odd.}
%\end{cases}\]
%\item If $u=(u_1,u_2,\dots,u_r)$ so that the concatenation $u_1\circ u_2,\circ\cdots\circ u_r$ is a subword of the infinite word $(v,v,w,w,v,v,w,w,\dot)$ with ${v,w}\in E$, $k$ $v$'s and $l$ $w$'s (in particular, $|k-l|\leq 4$) then 
%\[\Pi(u)=
%\begin{cases}
%\emptyset&\text{if both $k$ and $l$ are even,}\\
%v&\text{if $k$ is odd and $l$ even.}\\
%w &\text{if $k$ is even and $l$ odd.}\\
%(v,w) & if they are both odd
%\end{cases}\]
%\end{enumerate}

Finally, the inversion $(-)^{-1}$ is defined by sending a word $(x)$ to its symmetric words (i.e. for all $x\in V$, $(x)^{-1}=(x)$ and for all $\{x,y\}\in E$, $(x,y)^{-1}=(y,x)$).

\begin{prop}\label{prop:E(Gamma) is a partial group}
Let $\Gamma=(V,E)$ be a graph. the tuple $\left(\Ec(\Gamma),\Db(\Gamma),\Pi,(-)^{-1}\right)$ constructed above defines a partial group.
Moreover, this construction gives a functor $\Ec\colon \Graphs\to \Part$.
\end{prop}

\begin{proof}
By construction of $\Db(\Gamma)$, the conditions \ref{cond:D1} and \ref{cond:D2} of Definition \ref{def:PG} are satisfied. By construction of $\Pi$, the condition \ref{cond:P1} is satisfy. For condition \ref{cond:P2} a case-by-case study shows that it is also satisfied. For example, let $u,v,w\in \Db(\Gamma)$ and assume that $u\circ v \circ w\in \Db(\Gamma)$ is a word, involving two vertices $x,y\in V$. Assume also that $u,v$ and $w$ are non empty and that $v$ is of the form 
\[ (x,\underbrace{x,x,\dots,x}_{2k_1\text{ terms}},\underbrace{y,y,\dots,y}_{2r_1\text{ terms}},\underbrace{x,x,\dots,x}_{2k_2\text{ terms}},\underbrace{y,y,\dots,y}_{2r_2\text{ terms}},\dots, \underbrace{x,x,\dots,x}_{2k_n\text{ terms}},\underbrace{y,y,\dots,y}_{2r_n\text{ terms}})\]
with $n\geq 1$ and for all $j\in\{1,2,\dots,n\}$,  $k_j,r_j\geq 1$.
Then, since $u$ is not empty, by definition of $\Db(\Gamma)$ we can write $u=u'\circ (x)$ with $u'\in \Db(\Gamma)$. By definition of $\Pi$, we get that 
\begin{align*}
\Pi(u\circ v\circ w)&=\Pi(u'\circ (x)\circ v \circ w)\\
&=\Pi(u'\circ (\underbrace{x,x,\dots,x}_{2k_1+2\text{ terms}},\underbrace{y,y,\dots,y}_{2r_1\text{ terms}},\dots, \underbrace{x,x,\dots,x}_{2k_n\text{ terms}},\underbrace{y,y,\dots,y}_{2r_n\text{ terms}})\circ w)\\
&=\Pi(u'\circ w).
\end{align*}
We also have $\Pi(v)=(x)$ and thus $\Pi(u\circ\Pi(v)\circ w)=\Pi(u'\circ(x)\circ(x)\circ w)=\Pi(u\circ (x,x)\circ w)=\Pi(u'\circ w)$. In particular, $\Pi(u\circ v\circ w)=\Pi(u\circ\Pi(v)\circ w)$. The other cases can be managed in a similar fashion. 
Finally, the definition of the inversion, $\Db(\Gamma)$ and $\Pi$ ensure that condition \ref{cond:P3} is satisfied.

For the functoriality, notice that, if $\Gamma_1=(V_1,E_1)$ and  $\Gamma_2=(V_2,E_2)$ are two graphs and $f\colon \Gamma_1\to\Gamma_2$ is a graph homomorphism, the induced map $\Wb(f)\colon\Wb(V_1)\to\Wb(V_2)$ sends $\Ec(\Gamma_1)$ to $\Ec(\Gamma_2)$. This restriction gives a  partial groups homomorphism $\Ec(f)\colon\Ec(\Gamma_1)\to\Ec(\Gamma_2)$ and all this is compatible with composition of graph homomorphisms. 
% By construction of $\Pi$, the condition \ref{cond:P1} and the equality in \ref{cond:P2} are also satisfied. To check the property in the domain in \ref{cond:P2}, let $u,v,w\in \Db(\Gamma)$ and assume that $u\circ v\circ w\in\Db(\Gamma)$ (recall that here $\omega=u\circ v\circ w\in \Wb(\Wb(V))$). If the concatenation of the words in $\omega$ involves only a vertex then $u\circ\Pi(v)\circ w\in \Db(\Gamma)$ else, $\omega$ involves a ....
\end{proof}

The following Lemma is a direct consequence of the partial group structure on $\Ec(\Gamma)$ given above.

\begin{lem}\label{lem:properties of E(Gamma)}
Let $\Gamma=(V,E)$ be a graph.
\begin{enumerate}[label=(\roman*)]
\item\label{item:properties of E(Gamma)(i)} For all $u\in \Ec(\Gamma)$, $(u,u)\in \Db(\Gamma)$ if and only if there exists $x\in V$ such that $u=(x)$.
\item\label{item:properties of E(Gamma)(ii)} For all $x,y\in V$ with $x\neq y$, $((x),(y))\in \Db(\Gamma)$ if and only if $\{x,y\}\in E$. Moreover, if these properties are satisfied, then we have $(x,y)=\Pi((x),(y))$.
\item\label{item:properties of E(Gamma)(iii)} If $\Gamma$ is finite, then $\Ec(\Gamma)$ is also finite and $|\Ec(\Gamma)|=1+|V|+2|E|$.
\end{enumerate}
\end{lem}

We now look more closely at what happens on morphisms. %As in the proof of Proposition \ref{prop:E(Gamma) is a partial group}, for  $f\colon \Gamma_1\to\Gamma_2$ a graph homomorphism, we denote by $\Ec(f)\colon \Ec(\Gamma_1)\to\Ec(\Gamma_2)$ the associated morphism of partial groups.

\begin{lem}\label{lem:part to graph}
Let $\Gamma_1=(V_1,E_1)$ and $\Gamma_2=(V_2,E_2)$ be two finite graphs and $\varphi\colon\Mc(\Gamma_1)\to \Mc(\Gamma_2)$ a partial group homomorphism.
\begin{enumerate}[label=(\alph*)]
\item\label{item:part to graph(a)} $\varphi$ is totally determined by its image on $\{(x)\mid x\in V_1\}\subseteq \Ec(\Gamma_1)$.
\item\label{item:part to graph(b)} $\varphi\left(\{(x)\mid x\in V_1\}\right)\subseteq\{(x)\mid x\in V_2\}$.
\item\label{item:part to graph(c)} Assume $\varphi$ is injective and let $x,y\in V_1$ with $x\neq y$. If $\{x,y\}\in E_1$, then $\{\overline{\varphi}(x),\overline{\varphi}(y)\}\in E_2$ and $\varphi((x,y))=(\overline{\varphi}(x),\overline{\varphi}(y))$, where $\overline{\varphi}\colon V_1\to V_2$ is the application such that, for all $x\in V_1$, $\varphi((x))=((\overline{\varphi}(x))$ ($\overline{\varphi}$ exists by \ref{item:part to graph(b)}).
\end{enumerate}
\end{lem}

%Notice that we really need $\varphi$ to be an isomorphism in Lemma~\ref{lem:part to graph}.\ref{lem:(c)}. To see that, one can take, as an example, $\Gamma_1$ to be a graph with two vertices and one edge and $\Gamma_2$ to be a graph with only one vertex.

\begin{proof} 
Part \ref{item:part to graph(a)} follows from Lemma \ref{lem:properties of E(Gamma)}\ref{item:properties of E(Gamma)(ii)} since $\varphi$ is a partial group homomorphism. 

By Lemma \ref{lem:properties of E(Gamma)}\ref{item:properties of E(Gamma)(i)}, for all $i\in\{1,2\}$, the subset $\{(x)\mid v\in V_i\}\subseteq\Ec(\Gamma_i)$ is exactly the set of all the non trivial elements of $\Ec(\Gamma_i)$ that can be multiplied with themselves. Therefore $\varphi(\{(x)\mid x\in V_1\})\subseteq \{(x)\mid x\in V_2\}$ and \ref{item:part to graph(b)} follows.

Finally, assume that $\varphi$ is injective and let $x,y\in V_1$. Assume moreover that $\{x,y\}\in E_1$. Then, $(x,y)\in\Ec(\Gamma_1)$ and by Lemma \ref{lem:properties of E(Gamma)}\ref{item:properties of E(Gamma)(ii)} we have $((x),(y))\in\Db(\Gamma_1)$ and $(x,y)=\Pi((x),(y))$. Therefore $((\overline{\varphi}(x)),(\overline{\varphi}(y)))=(\varphi(x),\varphi(y))\in \Db(\Gamma_2)$. 
Now since $\varphi$ is injective, $\overline{\varphi}(x)\neq\overline{\varphi}(y)$ and, again by Lemma \ref{lem:properties of E(Gamma)}\ref{item:properties of E(Gamma)(ii)}, $\{\overline{\varphi}(x),\overline{\varphi}(y)\}\in\Ec(\Gamma_2)$ and $\Pi((\overline{\varphi}(x)),(\overline{\varphi}(y)))=(\overline{\varphi}(x),\overline{\varphi}(y))$. Finally, since $\varphi$ is a partial group homomorphism, $\varphi((x,y))=\varphi(\Pi((x),(y)))=\Pi(\varphi((x)),\varphi((y)))=\Pi((\overline{\varphi}(x)),(\overline{\varphi}(y)))=(\overline{\varphi}(x),\overline{\varphi}(y))$.   
\end{proof}

%Notice that, in the property \ref{lem:(c)},  depending if $v_2\leq w_2$ or $w_2\leq v_2$, we will have $\varphi(a_{e_1})=a_{e_2}$ or $\varphi(a_{e_1})=a_{e_2}^{-1}$.

%We are now ready to prove the main result. 

\begin{thm}\label{thm:graph to part}
Let $\Gamma=(V,E)$ be a graph.
Then $\Aut_\Graphs(\Gamma)=\Aut_\Part\left(\Ec(\Gamma)\right)$.
\end{thm} 

\begin{proof} 
By functoriality, $\Ec\colon \Graphs\to \Part$ induces a group homomorphism $\Phi\colon \Aut_\Graphs(\Gamma)\to\Aut_\Part\left(\Ec(\Gamma)\right)$.
 
% \undeprline{$\Phi$ is injective :}
 Now let $f$ be a automorphism of $\Gamma$ and assume that $\Phi(f)=\Id$. In particular, by Lemma \ref{lem:part to graph}\ref{item:part to graph(b)}, for every $x\in V$, $f(x)=\overline{\Phi(f)}(x)=x$. Thus $f=\Id_\Gamma$. 
 
 % \underline{$\Phi$ is surjective :}
 Finally, let $\varphi\in\Aut_\Part(\Ec(\Gamma))$ and set $\Ec_V=\{(x)\mid x\in V\}$. By Lemma \ref{lem:part to graph}\ref{item:part to graph(b)} and since $\varphi$ is an automorphism, for all $x\in V$, $\varphi((x))\in \Ec_V$ and $\varphi$ induces a bijection $\overline{\varphi}\colon V\to V$. Then, by Lemma~\ref{lem:part to graph}\ref{item:part to graph(c)}, $\overline{\varphi}\in\Aut_\Graphs(\Gamma)$. Finally, $\Phi(\overline{\varphi})=\varphi$ and so $\Phi$ is surjective.
\end{proof}

\subsection{From graph to posets}%%\url{https://arxiv.org/abs/2008.04997}

Finally, we construct a functor from $\Graphs$ to $\Posets$. What follows is not really knew and was already highlighted in the introduction of \cite{Ba}.

Let $\Gamma=(V,E)$ be a graph, we consider the poset $P(\Gamma)=\{\{x\}\mid x\in V\}\cup E$, i.e. $P(G)$ is a subset of $\Pc(V)$ the set of subsets of $V$ consisting in all the singleton in $V$ and the elements of $E$ (which are subset of $V$ of size 2),  order by inclusion. This poset is sometimes called the face poset of $\Gamma$ and it defines a functor $P\colon \Graphs\to \Posets$.

The following Lemma is a direct consequence of the construction.

\begin{lem}\label{lem:properties of P(Gamma)}
Let $\Gamma=(V,E)$ be a graph. We will denote by $\leq$ the partial order on $P(\Gamma)$.
\begin{enumerate}[label=(\roman*)]
\item\label{item:properties of P(Gamma)(i)} The set of maximal elements of $P(\Gamma)$ is $E$ together with the isolated vertices of $\Gamma$ and the set of minimal element of $P(\Gamma)$ is $V$.
\item\label{item:properties of P(Gamma)(ii)} For all $u,v\in P(\Gamma)$ with $u\neq v$, $u\leq v$ if and only if $v\in E$ and there exists $x\in V$ such that $u=\{x\}$ and $x$ is one of the extremities of $v$.
\item\label{item:properties of P(Gamma)(iii)} If $\Gamma$ is a finite, then $P(\Gamma)$ is also finite and $|P(\Gamma)|=|V|+|E|$.
\end{enumerate}
\end{lem}

Let us now look at what happens on isomorphisms.

\begin{lem}\label{lem:poset to graph}
Let $\Gamma_1=(V_1,E_1)$ and $\Gamma_2=(V_2,E_2)$ be two finite graphs and $\varphi\colon P(\Gamma_1)\to P(\Gamma_2)$ a poset homomorphism and assume that $\varphi$ is an isomorphism.
\begin{enumerate}[label=(\alph*)]
\item\label{item:poset to graph(a)} $\varphi(\{\{x\}\mid x\in V_1\})= \{\{x\}\mid x\in V_2\}$ and $\varphi(E_1)= E_2$.
\item\label{item:poset to graph(b)} Let $x,y\in V_1$ with $x\neq y$. If $\{x,y\}\in E_1$, then $\{\overline{\varphi}(x),\overline{\varphi}(y)\}\in E_2$ and $\varphi(\{x,y\})=\{\overline{\varphi}(x),\overline{\varphi}(y)\}$, where $\overline{\varphi}\colon V_1\to V_2$ is the application such that, for all $x\in V_1$, $\varphi(\{x\})=\{\overline{\varphi}(x)\}$ ($\overline{\varphi}$ exists by \ref{item:poset to graph(a)}). 
\end{enumerate}
\end{lem}

\begin{proof}
Since $\varphi$ is an isomorphism, then it sends the minimal, resp. maximal, elements of $P(\Gamma_1)$ to the the minimal, resp. maximal, elements of $P(\Gamma_2)$. Therefore, part \ref{item:poset to graph(a)} is consequence of Lemma \ref{lem:properties of P(Gamma)}\ref{item:properties of P(Gamma)(i)}. 

Now, let  $x,y\in V_1$ with $x\neq y$. If $\{x,y\}\in E_1$, by \ref{item:poset to graph(a)}, $\varphi(\{x,y\})\in E_2$. So there exists $a,b\in V_2$ such that $\{a,b\}\in E_2$ and $\varphi(\{x,y\})=\{a,b\}$. Since $\{x\}\leq \{x,y\}$ and $\{y\}\leq \{x,y\}$ and $\varphi$ is a poset homomorphism, $\varphi(\{x\})=\{\overline{\varphi}(x)\}\subset \{a,b\}$ and $\varphi(\{y\})=\{\overline{\varphi}(y)\}\subset \{a,b\}$. in particular, by \ref{lem:properties of P(Gamma)}\ref{item:properties of P(Gamma)(ii)}, $\overline{\varphi}(x)\in\{a,b\}$ and $\overline{\varphi}(y)\in\{a,b\}$. As $\varphi$ is injective and $x\neq y$, we have $\{\overline{\varphi}(x),\overline{\varphi}(y)\}=\{a,b\}$ and part \ref{item:poset to graph(b)} follows.
\end{proof}

\begin{thm}\label{thm:graph to poset}
Let $\Gamma=(V,E)$ be a graph.
Then $\Aut_\Graphs(\Gamma)=\Aut_\Posets\left(P(\Gamma)\right)$.
\end{thm}

\begin{proof}
By functoriality, the functor $P\colon \Graphs\to \Posets$ induces a group homomorphism $\Phi\colon \Aut_\Graphs(\Gamma)\to\Aut_\Posets\left(P(\Gamma)\right)$.
 
Now, let $f$ be a automorphism of $\Gamma$ and assume that $\Phi(f)=\Id$. By Lemma~\ref{lem:poset to graph}\ref{item:poset to graph(a)}, $f=\overline{\Phi(f)}=\Id$ and so $\Phi$ is injective.  
 
 Finally, let $\varphi\in\Aut_\Posets(\Ec(\Gamma))$ and set $P_V=\{(x)\mid x\in V\}$. By Lemma~\ref{lem:poset to graph}\ref{item:poset to graph(a)} f induces a bijection $\overline{\varphi}\colon V\to V$ and, by Lemma~\ref{lem:poset to graph}\ref{item:poset to graph(b)}, $\overline{\varphi}\in\Aut_\Graphs(\Gamma)$. Moreover, $\Phi(\overline{\varphi})=\varphi$ and so $\Phi$ is surjective.
\end{proof}

\section{Some particular graphs}\label{sec:graphs}

We define here the graphs that will be the main tool and describe their automorphisms.

Fix $p,q\geq 1$ and set $\Gamma_{p,q}=(V_{p,q},E_{p,q})$ where $V_{p,q}$ is a set with $p+q+1$ elements denoted 
\[V_{p,q}=\{c_0,c_1,c_2,\dots,c_p,d_1,d_2,\dots,d_q\},\] 
and where
\[E_{p,q}=\{\{(c_0,c_i\}\mid i\in\{1,2,\dots,p\}\}\cup\{\{c_0,d_1\}\}\cup\{\{d_i,d_{i+1}\}\mid i\in\{1,2,\dots,q-1\}\}.\]

The graph $\Gamma_{p,q}$ can be though as a star with $p$ spikes and a tail of length $q$. For example here are $\Gamma_{2,3}$ and $\Gamma_{3,2}$.

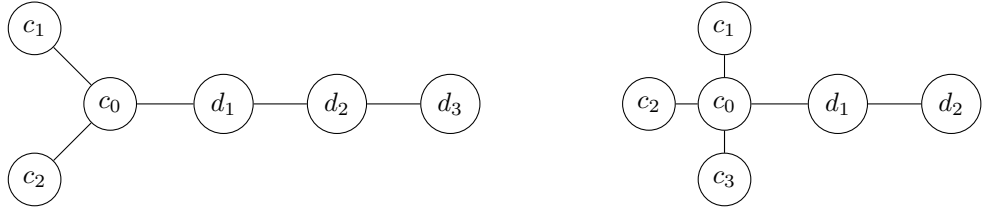
\begin{center}
\hfill
\begin{tikzpicture}
\node[shape=circle,draw=black] (C1) at (0,1) {$c_1$};
\node[shape=circle,draw=black] (C2) at (0,-1) {$c_2$};
\node[shape=circle,draw=black] (C0) at (1,0) {$c_0$};
\node[shape=circle,draw=black] (D1) at (2.5,0) {$d_1$};
\node[shape=circle,draw=black] (D2) at (4,0) {$d_2$};
\node[shape=circle,draw=black] (D3) at (5.5,0) {$d_3$};

\path (C1) edge (C0);
\path (C2) edge (C0);
\path (C0) edge (D1);
\path (D1) edge (D2);
\path (D2) edge (D3);
\end{tikzpicture}
\hfill
\begin{tikzpicture}
\node[shape=circle,draw=black] (C1) at (1,1) {$c_1$};
\node[shape=circle,draw=black] (C2) at (0,0) {$c_2$};
\node[shape=circle,draw=black] (C3) at (1,-1) {$c_3$};
\node[shape=circle,draw=black] (C0) at (1,0) {$c_0$};
\node[shape=circle,draw=black] (D1) at (2.5,0) {$d_1$};
\node[shape=circle,draw=black] (D2) at (4,0) {$d_2$};

\path (C1) edge (C0);
\path (C2) edge (C0);
\path (C0) edge (C3);
\path (C0) edge (D1);
\path (D1) edge (D2);
\path (C0) edge (C3);
\end{tikzpicture}
\hfill${}$
\end{center}

Finally we set $\Gamma_{p,q}^+=\Gamma_{p,q}\cup\{c_0'\}$ to be the graph $\Gamma_{p,q}$ together with a single isolated vertex added.

\begin{prop}\label{prop:aut Gamma p q}
Let $p,q\geq 1$, then $|V_{p,q}|=p+q+1$ and $|E_{p,q}|=p+q$.
If $p,q\geq 2$ then the group $\Aut_\Graphs(\Gamma_{p,q})\cong\Aut_\Graphs(\Gamma_{p,q}^+)\cong\Sigma_{p}$. 
\end{prop}

\begin{proof} 
The cardinalities of the vertex and edge sets are a direct consequence of the definition. 

Assume now that $p,q\geq 2$.
Notice first that, since $c_0'$ is the unique isolated vertex in $\Gamma_{p,q}^+$, it is fixed by any automorphism and so $\Aut_\Graphs(\Gamma_{p,q}^+)\cong\Aut_\Graphs(\Gamma_{p,q})$. It remains to prove that $\Aut_\Graphs(\Gamma_{p,q})\cong \Sigma_p$.
 
As $c_0$ is the only vertex of degree 3 or more, it is fixed by any automorphism. Then the tail $c_0 - d_1 - d_2 - \cdots - d_q$, which is the only tail starting at $c_0$ of length 2 or more, is also fixed by any automorphism. Therefore, $\Aut_\Graphs(\Gamma_{p,q})$ acts faithfully on $\{c_1,c_2,\dots, c_p\}$ and this action gives an isomorphism $\Aut_\Graphs(\Gamma_{p,q})\cong \Sigma_p$. 
\end{proof}

In particular the ratio $\dfrac{|\Aut_\Graphs(\Gamma_{p,q})|}{|V_{p,q}|}$ is equal to $\dfrac{p!}{p+q+1}$ and we have the following corollary.

\begin{cor}\label{cor:main graphs}
Let $r\in\Qb\cap]0,+\infty[$. There exists a graph $\Gamma=(V,E)$ such that $r=\dfrac{|\Aut_\Graphs(\Gamma)|}{|V|}$.
\end{cor}

\begin{proof}
take $m,n\geq 3$ such that $r=\dfrac{m}{n}$. Then taking $p=m$ and $q=n\times (m-1)!-m-1$, we have $p\geq 2$, $q\geq 3(m-1)-m-1=2m-4\geq 2$ and thus 
\[\frac{|\Aut_\Graphs(\Gamma_{p,q})|}{|V_{p,q}|}=\frac{p!}{p+q+1}=\frac{m!}{n\times (m-1)!}=\frac{m}{n}=r.\]
\end{proof}

Now if we consider the graph $\Gamma_{p,q}^+$ for $p$ and $q$ well chosen to get the following two corollaries. 

\begin{cor}\label{cor:main graphs2}
Let $r\in\Qb\cap]0,+\infty[$. There exists a graph $\Gamma=(V,E)$ such that $r=\dfrac{|\Aut_\Graphs(\Gamma)|}{|V|+|E]}$.
\end{cor}

\begin{proof}
take $m\geq 3$ and $n\geq 2$ such that $r=\dfrac{m}{n}$. Then taking $p=2m$ and $q=n\times (2m-1)!-2m-1$, we have $p\geq 6>2$, $q\geq 2(2m-1)-2m-1=2m-3>2$ and thus the wanted quotient is, 
\[\frac{|\Aut_\Graphs(\Gamma_{p,q}^+)|}{|V_{p,q}|+1+|E_{p+q}|}=\frac{p!}{2(p+q+1)}=\frac{(2m)!}{2\times n\times (2m-1)!}=\frac{2m}{2n}=\frac{m}{n}=r.\]
\end{proof}

\section{Applications to algebraic structures}\label{sec:applications}

We give here the main results of the paper which are consequences of two the previous sections.

\begin{thm}\label{thm: main evoalg}
Let $k$ be a field, and $r\in\Qb\cup ]0,\infty[$. Then there exists a
regular evolution $k$-algebra $A$ such that $r=\frac{|\Aut_{\kAlg}(A)|}{\dim_k(A)}$.
\end{thm}

\begin{proof}
By Corollary \ref{cor:main graphs2}, there exists a graph $\Gamma=(V,E)$ such that $r=\dfrac{|\Aut_\Graphs(\Gamma)|}{|V|+|E]}$. Also, by Theorem \ref{thm:graph to evoalg}, there exists a regular evolution $k$-algebra $A$ such that $\dim_k(A)=|E|+|V|$ and  $\Aut_{Graphs}(\Gamma)\cong\Aut_{\kAlg}(A)$. Hence 
\[r=\frac{|\Aut_\Graphs(\Gamma)|}{|V|+|E]}=\frac{|\Aut_{\kAlg}(A)|}{\dim_k(A)}.\]
\end{proof}

\begin{rem}
Following Remark \ref{rmk:realisability evoalg}, one could have get the same result with the construction of Sriwongsa and Zou \cite{SZ} using Corollary \ref{cor:main graphs} instead of Corollary \ref{cor:main graphs}. 
\end{rem}

\begin{thm}\label{thm:main monoids}
Let $r\in\Qb\cap]0,+\infty[$. There exists a monoid $M$ such that $r=\dfrac{|\Aut_\Monoids(M)|}{|M|}$. Moreover, the monoid $M$ can be chosen to be commutative. 
\end{thm}

\begin{proof}
take $m\geq 3$ and $n\geq 2$ such that $r=\dfrac{m}{n}$. Then taking $p=2m$ and $q=n\times (2m-1)!-2m-2$, we have $p\geq 6 > 2$, $q\geq 2(2m-1)-2m-2=2m-4\geq 2$ and thus, by Theorem \ref{thm:graph to mon} and Proposition \ref{prop:aut Gamma p q}, \[|\Aut_\Monoids(M(\Gamma_{p,q}))|=|\Aut_\Graphs(\Gamma_{p,q})|=p!=(2m)!.\] Also by Lemma \ref{lem:properties of M(Gamma)}\ref{item:properties of M(Gamma)(v)},
  \[|M(\Gamma_{p,q})|=|V_{p,q}|+|E_{p,q}|+2=2(p+q)+4=2n\times (2m-1)!.\]
Therefore,
\[\frac{|\Aut_\Monoids(M(\Gamma_{p,q}))|}{|M(\Gamma_{p,q})|}=\frac{(2m)!}{2n\times (2m-1)!}=\frac{2m}{2n}=\frac m n=r.\]
Finally, recall that by Lemma \ref{lem:properties of M(Gamma)}\ref{item:properties of M(Gamma)(i)}, $M(\Gamma_{p,q})$ is commutative.
\end{proof}

\begin{thm}\label{thm:main part}
Let $r\in\Qb\cap]0,+\infty[$. There exists a partial group $\Mc$ such that $r=\dfrac{|\Aut_\Part(\Mc)|}{|\Mc|}$.
\end{thm}

\begin{proof}
take $m,n\geq 2$ such that $r=\dfrac{m}{n}$. Then taking $p=3m$ and $q=n\times(3m-1)!-3m-1$, we have $p\geq 6> 2$ and $q\geq 2(3m-1)-3m-1=3m-3> 2$ and thus , by Theorem \ref{thm:graph to part} and Proposition \ref{prop:aut Gamma p q}, \[|\Aut_\Part(\Ec(\Gamma_{p,q}^+)|=|\Aut_\Graphs(\Gamma_{p,q}^+)|=p!=(3m)!.\] Also by Lemma \ref{lem:properties of E(Gamma)}\ref{item:properties of E(Gamma)(iii)}, 
\[|\Ec(\Gamma_{p,q}^+))|=(|V_{p,q}|+1)+2|E_{p,q}|+1=(p+q+2)+2(p+q)+1=3(p+q+1)=3n\times(3m-1)!.\]
Therefore,
\[\frac{|\Aut_\Part(\Ec(\Gamma_{p,q}^+))|}{|\Ec(\Gamma_{p,q}^+)|}=\frac{(3m)!}{3\times n\times(3m-1)!}=\frac{3m}{3n}=\frac{m}{n}=r.\]
\end{proof}

\begin{thm}\label{thm:main posets}
Let $r\in\Qb\cap]0,+\infty[$. There exists a poset $P$ such that $r=\dfrac{|\Aut_\Posets(P)|}{|P|}$.
\end{thm}

\begin{proof}
take $m\geq 3$ and $n\geq 2$ such that $r=\dfrac{m}{n}$. Then taking $p=2m$ and $q=n\times (2m-1)!-2m-1$, we have $p\geq 6> 2$, $q\geq 2(2m-1)-2m-1=2m-3\geq3> 2$ and thus, using  Theorem \ref{thm:graph to poset}, Proposition \ref{prop:aut Gamma p q} and Lemma \ref{lem:properties of P(Gamma)}\ref{item:properties of P(Gamma)(iii)},
\[\dfrac{|\Aut_\Posets(P(\Gamma_{p,q}^+))|}{|P(\Gamma_{p,q}^+)|}=\frac{|\Aut_\Graphs(\Gamma_{p,q}^+)|}{|P(\Gamma_{p,q}^+)|}=\frac{p!}{2(p+q)+2}=\frac{2m!}{2n\times(2m-1)!}=\frac{2m}{2n}=\frac{m}{n}=r.\]
\end{proof}

%%% Biblio

\bibliographystyle{abbrv}
\bibliography{biblio}

\end{document}